\documentclass{article}
\usepackage{amsfonts}
\usepackage{amsmath}
\usepackage{subfigure}
\usepackage{color}
\usepackage{soul}
\usepackage{amssymb}
\usepackage{amsthm}
\usepackage{thmtools}

\newtheorem{theorem}{Theorem}[section]

\newtheorem{lemma}[theorem]{Lemma}
\newtheorem{proposition}[theorem]{Proposition}

\newtheorem{definition}{Definition}

{
	\newtheorem{assumption}{A}
	\newtheorem{condition}{C}
}

\def\blfootnote{\xdef\@thefnmark{}\@footnotetext}

\title{{\itshape Center Manifolds of Differential Equations in Banach Spaces}}
\author{Y.-M. Chung
\\University of North Carolina at Greensboro
\\North Carolina,
United States, \\ and \\E. Schaal 
\\College of William and Mary, Virgina,
United States}

\begin{document}

Running Header: {\em{CENTER MANIFOLDS IN BANACH SPACES}}

Corresponding Author: Emily Schaal, eeschaal@email.wm.edu

\clearpage

\maketitle

\begin{abstract}

The center manifold is useful for describing the long-term behavior of a system of differential equations.  In this work, we consider an autonomous differential equation in a Banach space that has the exponential trichotomy property in the linear terms and Lipschitz continuity in the nonlinear terms.  Using the {\it spectral gap condition} we prove the existence and uniqueness of the center manifold. Moreover, we prove the regularity of the manifold with a few additional assumptions on the nonlinear term.  We approach the problem using the well-known Lyapunov-Perron method, which relies on the Banach fixed-point theorem.  The proofs can be generalized to a non-autonomous system.\footnotetext{34K19; 37L10; 65L10}

\end{abstract}

\section{Introduction}

The center manifold,  first introduced by Pliss
\cite{pliss} and Kelley \cite{kelley}, can roughly be defined as the steady states of the differential equations around which the behavior of the trajectories near enough to it will never be governed by neither the unstable nor the stable manifolds.  Another way of describing it is the set of initial conditions whose trajectories are bounded both forward and backward in time.  There are a number of works that research the existence and uniqueness of the center manifold: for example, in \cite{capiski}, Capiski presents a rigorous computer-assisted proof on the three-body problem; Chow and Liu \cite{chow} use the Hadamard graph transform method; and \cite{fenichel, hirsch, shoshitaishvili} discuss studies of arbitrarily smooth local center manifolds. The works of \cite{carr, edneral, vanderbauwhede} consider the manifold in $\mathbb{R}^n$ in the context of studying different varieties of systems. Alternatively, \cite{chicone, kelley, shashkov} study the global manifold using an integral operator in a Banach space with a general approach that makes use of minimal structure and assumptions.  Similar concepts are found in different studies, such as that of slow manifolds in both random and deterministic dynamical systems (see e.g. Lorenz in \cite{lorenz} and Roberts in \cite{roberts}) and inertial manifolds in fluid dynamics (see e.g. \cite{foias, temam} and references therein).

One of main applications is called the {\it center manifold reduction}.  Since the long term behavior of the system is contained in the center manifold, one can restrict the system on the manifold to obtain a system of lower dimensions that has fundamentally the same long term behavior as the original system.  To ensure the reduction exists, the manifold needs to be smooth enough that the solutions to the lower dimensional system exist.  Hence, the regularity of the manifold is fundamental; see e.g. \cite{sandstede}.  There are also developments for different methods of computing center manifolds, see \cite{jolly, fuming, hamzi, dellnitz, ito}. Researchers study bifurcation analysis on such reduced system; see e.g.  \cite{ferrara, liaw}.  

In this paper, we study the existence, uniqueness, and regularity of the center manifold by the Lyapunov-Perron method.  Like previous works, we consider the differential equation in a Banach space and make use of an integral operator, called {\it Lyapunov-Perron} operator.  The main contribution of this work is twofold.  First, we establish all the proofs outlined by Jolly and Rosa in \cite{jolly} where a numerical method for computing the manifold is presented. We modify the framework of \cite{jolly} to the context of ordinary differential equations in Banach space. Second, the idea of the proofs provides a basis for a simple algorithm to compute the center manifold (see authors forthcoming work \cite{schaal}).  The paper is organized as follows. In Section \ref{sec:framework}, the framework, notations, and assumptions are discussed; we also introduce the Lyapunov-Perron operator. In Section \ref{sec: existence}, we establish the existence and uniqueness of the manifold. In Section \ref{sec:regularity of the manifold}, we make some extensions to the framework and follow the same line of proof for the derivative of the map. We prove that the map whose graph gives the center manifold is $\mathcal{C}^1$. This proof is the first step in the induction to show that the center manifold is $\mathcal{C}^k$ for general $k$.

\section{Framework}
\label{sec:framework}
Consider a nonlinear ordinary differential systems in a Banach space $E$ that can be decomposed as such: $E = X \times Y \times Z$, where $u \in E$ takes the form $u = x + y + z$. The space has associated norm $\|u\| = \max\{\|x\|,\|y\|,\|z\| \}$:
\begin{equation}\label{bigsystem}
\begin{aligned}
\dot{x} &= Ax + F(u)  \\
\dot{y} &= By + G(u)  \\
\dot{z} &= Cz + H(u).
\end{aligned}
\end{equation}
We have that $A \in \mathcal{L}(X,X)$, $B \in \mathcal{L}(Y,Y)$, and $C \in \mathcal{L}(Z,Z)$, where $\mathcal{L}$ is the space of linear operators; we also have $F(u) \in \mathcal{C}(E,X),$ $G(u) \in \mathcal{C}(E,Y)$, and $H(u) \in \mathcal{C}(E,Z)$. We make the following assumptions.
\begin{assumption}{Exponential Trichotomy Condition:}\label{assumption: exponential trichotomy}
	For $\alpha_x$, $\alpha_y$, $\beta_y$, $\beta_z$, $K_x$, $K_y$, and $K_z \in \mathbb{R}$ with ordering $\alpha_x < \alpha_y \leq \beta_y < \beta_z$ and $t \geq 0$,
	\begin{align}
	\|e^{tA}\| &\leq K_x e^{\alpha_x t}  &
	\|e^{-tC}\| &\leq K_z e^{-\beta_z t}\\
	\|e^{-tB}\| &\leq K_ye^{-\beta_y t} & \|e^{tB}\| &\leq K_y e^{\alpha_y t}.
	\end{align}
\end{assumption}
\begin{assumption}{Lipschitz Continuity of Nonlinear Terms:}\label{assumption: nonlinear lipschitz}
	For $u_1$ and $u_2 \in E$, there exist constants $\delta_x$, $\delta_y$, and $\delta_z \in \mathbb{R}_{> 0}$ and
	\begin{align}
	\|F(u_1) - F(u_2)\| &\leq \delta_x\|u_1 - u_2\|  \\
	\|G(u_1) - G(u_2)\| &\leq \delta_y\|u_1 - u_2\|  \\
	\|H(u_1) - H(u_2)\| &\leq \delta_z\|u_1 - u_2\|. 
	\end{align}
\end{assumption}
\begin{assumption}{Gap Condition:}\label{assumption: gap condition}
	Given A\ref{assumption: exponential trichotomy} and A\ref{assumption: nonlinear lipschitz}, the following inequalities hold
	\begin{align}
	\beta_y - \alpha_x &> K_x \delta_x + K_y \delta_y \\
	\beta_z - \alpha_y &> K_y \delta_y + K_z \delta_z. 
	\end{align}
\end{assumption}

A\ref{assumption: exponential trichotomy} defines bounds for the linear parts of each component. The stable component is bounded forward in time, the unstable component is bounded backward in time, and the center is bounded in both directions. This is a generalization of the exponential dichotomy condition. By classical results, A\ref{assumption: nonlinear lipschitz} guarantees there exists a unique solution to the ordinary differential equation denoted by $u(t,u_0)$. A\ref{assumption: gap condition} is the main assumption on the invariant manifold. For investigations of the gap condition that we use, see \cite{castaneda, latushkin}. These three assumptions allow us to study the behavior of the invariant manifold as the global behavior of the system. 

We follow Jolly and Rosa \cite{jolly} and define a parameter $\sigma(t)$ such that 
\begin{align}\label{eq:sigma}
\sigma(t) = 
\begin{cases}
\sigma_p & t \geq 0, \\
\sigma_n & t \leq 0
\end{cases}
\end{align}
and define the following ordering conditions with respect to the constants in A\ref{assumption: exponential trichotomy} and A\ref{assumption: nonlinear lipschitz}:
\begin{condition}{Relation of Constants:} \label{c1}
	\begin{align}
	\alpha_x < \sigma_n < \alpha_y \leq \beta_y < \sigma_p < \beta_z,
	\end{align}
\end{condition}
\begin{condition}{Choice of $\sigma$:}\label{c2} 
	\begin{align}
	\alpha_x + K_x\delta_x < \sigma_n < \beta_y - K_y\delta_y \\
	\alpha_y + K_y\delta_y < \sigma_p < \beta_z - K_z\delta_z.
	\end{align}
\end{condition}
We define a function space $\mathcal{F}_\sigma$ such that each global trajectory $\phi$ of the differential system where $\phi: Y \rightarrow E$ is found as a fixed point in
\begin{align}
\mathcal{F}_\sigma = \{\phi \in \mathcal{C}(\mathbb{R},E):  \smash{\displaystyle\sup_{t \in \mathbb{R}}} (e^{-\sigma(t)}\|\phi(t)\|) = \|\phi\|_\sigma < \infty\}.
\end{align}
This is the space of all continuous functions from $\mathbb{R}$ to the space $E$ that are exponentially bounded, and it is these functions that we wish to study. $\mathcal{F}_\sigma$ is also a Banach space with the $\|\cdot\|_\sigma$ norm.

Finally, let $y_0 \in Y$, $\phi(t,y_0):=\phi(t) \in \mathcal{F}_\sigma$ and define the Lyapunov-Perron operator $\mathcal{T}: \mathcal{F}_\sigma \times Y \rightarrow \mathcal{F}_\sigma$ as
\begin{align}
\label{ali:fixedptmap}
\mathcal{T}(\phi(t), y_0) = \underbrace{e^{tB}y_0 + \int_0^t e^{(t-s)B}G(\phi(s))ds}_\text{I} - 
\underbrace{\int_t^\infty e^{(t-s)C}H(\phi(s))ds}_\text{II} \nonumber \\
+ \underbrace{\int_{-\infty}^t e^{(t-s)A}F(\phi(s))ds}_\text{III},
\end{align}
where I is the $Y$ component, II is the $Z$ component, and III is the $X$ component. 

Finally, we introduce some shorthand for $\beta_y$ and $\alpha_y$:
\begin{align}
c(t) =
\begin{cases}
\beta_y & \mbox{ when } t \leq 0 \\
\alpha_y & \mbox{ when } t \geq 0.
\end{cases}
\end{align} We discuss the two cases often, so this allows us to keep the proofs concise. 

\section{Existence of the Center Manifold}\label{sec: existence}	

The main construction of the center manifold in this work is to show that the manifold is a graph of some Lipschitz function $\Phi: Y \rightarrow X\times Z$.  In order words, $X$ (stable) and $Z$ (unstable) components can be represented by the $Y$ component on the manifold.  We begin our construction by establishing the following estimate which is essential throughout this paper.

\begin{lemma}\label{lemma1}
    Let \begin{align} 
    \delta_\phi = \max\{&\sup_{t\in\mathbb{R}}e^{(\alpha_x-\sigma(t)t)}\int_{-\infty}^t K_x\delta_x e^{(\sigma(s)-\alpha_x)s}ds, \label{case3} \\ &\sup_{t\in\mathbb{R}}e^{(c(t)-\sigma(t)t)}\int_0^t K_y\delta_ye^{(\sigma(s)-c(t))s}ds, \label{case1} \\ &\sup_{t\in\mathbb{R}}e^{(\beta_z-\sigma(t)t)}\int_t^\infty K_z\delta_ze^{(\sigma(s)-\beta_z)s}ds\}. \label{case2}
    \end{align}
    Then, by C\ref{c1} and C\ref{c2}, $\delta_\phi = \max\{\frac{K_y\delta_y}{\beta_y -\sigma_n}, \frac{K_y\delta_y}{\sigma_p-\alpha_y}, \frac{K_z\delta_z}{\beta_z -\sigma_p}, \frac{K_x\delta_x}{\sigma_n - \alpha_x} \} < 1.$
\end{lemma}

\begin{proof}
    We begin with \eqref{case1}. In this integral, $s$ and $t$ take on the same sign, so we can evaluate to get 
    \begin{align}\label{int1}K_y \delta_y\sup_{t\in\mathbb{R}}e^{(c(t)-\sigma(t))t}\bigg[\frac{e^{(\sigma(t)-c(t))s}}{\sigma(t)-c(t)}\bigg]_0^t.\end{align} 
    Note that we have
    \begin{align}
    c(t) - \sigma(t) = \beta_y - \sigma_n > 0 & \mbox{ when } t \leq 0 \label{tleq0} \\
    c(t) - \sigma(t) = \alpha_y - \sigma_p < 0 & \mbox{ when } t \geq 0; \label{tgeq0} 
    \end{align}
    this is a result of C\ref{c1}. For $t \leq 0$, \eqref{int1} evaluates to  
    \begin{align}\label{int2}K_y \delta_y\sup_{t\leq0}\bigg[-\frac{e^{(\beta_y-\sigma_n)t}}{\beta_y-\sigma_n}+\frac{1}{\beta_y-\sigma_n}\bigg]. \end{align} By  \eqref{tleq0}, the supremum for \eqref{int2} is reached when $t = 0$ and we get $\frac{K_y\delta_y}{\beta_y -\sigma_n}$. For $t \geq 0$, \eqref{int1} evaluates to \begin{align}\label{int3}K_y \delta_y\sup_{t\geq0}\bigg[\frac{1}{\sigma_p-\alpha_y}-\frac{e^{(\alpha_y-\sigma_p)t}}{\sigma_p-\alpha_y}\bigg].\end{align} By \eqref{tgeq0}, the supremum for \eqref{int3} is reached when $t = 0$ and we get $\frac{K_y\delta_y}{\sigma_p-\alpha_y}$.
    
    Moving on to \eqref{case2}, we can evaluate to get
    \begin{align}
    K_z \delta_z\sup_{t\leq 0}e^{(\beta_z-\sigma_n)t}\bigg(\bigg[\frac{e^{(\sigma_n-\beta_z)s}}{\sigma_n-\beta_z}\bigg]_t^0+\bigg[\frac{e^{(\sigma_p-\beta_z)s}}{\sigma_p-\beta_z}\bigg]_0^\infty\bigg) &\mbox{ when } t \leq 0 \label{int4} \\
    K_z \delta_z\sup_{t\geq 0}e^{(\beta_z-\sigma_p)t}\bigg[\frac{e^{(\sigma_p-\beta_z)s}}{\sigma_p-\beta_z}\bigg]_t^\infty  &\mbox{ when } t\geq 0 \label{int5}.
    \end{align} Note that $\sigma_p - \beta_z < 0$, so \eqref{int4} simplifies to
    \begin{align} \label{int6}
    &K_z \delta_z\sup_{t\leq 0}\bigg(\frac{e^{(\beta_z-\sigma_n)t}}{\sigma_n-\beta_z}-\frac{1}{\sigma_n-\beta_z}-\frac{e^{(\beta_z-\sigma_n)t}}{\sigma_p-\beta_z}\bigg) \nonumber \\
    = &K_z \delta_z\sup_{t\leq 0}\bigg(\frac{(\sigma_p-\sigma_n)e^{(\beta_z-\sigma_n)t}-(\sigma_p-\beta_z)}{(\sigma_p-\beta_z)(\sigma_n-\beta_z)}\bigg)
    \end{align}
    which has its supremum when $t = 0$ and we obtain $\frac{K_z \delta_z}{\beta_z-\sigma_p}$. Next, \eqref{int6} simplifies directly to $
    \frac{K_z \delta_z}{\beta_z-\sigma_p}$.
    
    Moving on to \eqref{case3}, we can evaluate to get 
    \begin{align}
    K_x \delta_x\sup_{t\leq 0}e^{(\alpha_x -\sigma_n)t}\bigg[\frac{e^{(\sigma_n-\alpha_x)s}}{\sigma_n-\alpha_x}\bigg]_{-\infty}^t  &\mbox{ when } t\leq 0 \label{int8} \\
    K_x \delta_x\sup_{t\geq 0}e^{(\alpha_x-\sigma_p)t}\bigg(\bigg[\frac{e^{(\sigma_n-\alpha_x)s}}{\sigma_n-\alpha_x}\bigg]_{-\infty}^0+\bigg[\frac{e^{(\sigma_p-\alpha_x)s}}{\sigma_p-\alpha_x}\bigg]_0^t\bigg) &\mbox{ when } t \geq 0 \label{int7}.
    \end{align}
    Note that $\sigma_n - \alpha_x > 0$ by C\ref{c1}, so \eqref{int8} simplifies directly to $\frac{K_x\delta_x}{\sigma_n - \alpha_x}$. Then, \eqref{int7} simplifies to
    \begin{align}
    &K_x \delta_x\sup_{t\geq 0}\bigg(\frac{e^{(\alpha_x-\sigma_p)t}}{\sigma_n-\alpha_x}+\frac{1}{\sigma_p-\alpha_x}-\frac{e^{(\alpha_x-\sigma_p)t}}{\sigma_p-\alpha_x}\bigg) \\
    =&K_x \delta_x\sup_{t\geq 0}\bigg(\frac{(\sigma_p-\sigma_n)e^{(\alpha_x-\sigma_p)t}+(\sigma_n-\alpha_x)}{(\sigma_p-\alpha_x)(\sigma_n-\alpha_x)}\bigg).
    \end{align}
    This has its supremum when $t=0$ and we obtain $\frac{K_x\delta_x}{\sigma_n - \alpha_x}$. Each of $\frac{K_y\delta_y}{\beta_y -\sigma_n}, \frac{K_y\delta_y}{\sigma_p-\alpha_y}, \frac{K_z\delta_z}{\beta_z -\sigma_p}, \frac{K_x\delta_x}{\sigma_n - \alpha_x}$ is less than one as a result of C\ref{c2}.
\end{proof}

The key to the construction is that the map $\mathcal{T}$ is a contraction mapping, so the fixed point exists by the Banach fixed point Theorem, and the center manifold can be found in that fixed point.  In the first step, we need to show that the $\mathcal{T}$ map is well defined.

\begin{proposition}\label{prop: well-defined}
	Assume A\ref{assumption: exponential trichotomy}, A\ref{assumption: nonlinear lipschitz}, and C\ref{c1}. Let $y_0 \in Y$ and $\phi \in \mathcal{F}_\sigma$, then $$\mathcal{T}(\phi(t),y_0) \in \mathcal{C}(\mathbb{R}, E) \mbox{ and } \|\mathcal{T}(\phi,y_0)\|_\sigma < \infty.$$ 
\end{proposition}
\begin{proof}
	First, we show that $\|\mathcal{T}(\phi,y_0)\|_\sigma < \infty$. If we take the norm of \eqref{ali:fixedptmap}, apply assumptions A\ref{assumption: exponential trichotomy} and A\ref{assumption: nonlinear lipschitz}, multiply by $e^{\sigma(s) s} e^{-\sigma(s) s}$, multiply by $e^{-\sigma(t) t}$, and take the supremum over $s$ in each integral:
	\begin{align}
	e^{-\sigma(t) t}\|\mathcal{T}(\phi(t),y_0)\| \leq \max\bigg\{&K_x \delta_x \|\phi\|_\sigma e^{(\alpha_x-\sigma(t)) t} \int_{-\infty}^{t}e^{(\sigma(s)-\alpha_x) s}ds, \\
	K_ye^{(c(t) - \sigma(t)) t}|y_0| + &K_y \delta_y \|\phi\|_\sigma e^{(c(t) - \sigma(t)) t}\int_{0}^{t}e^{(\sigma(s)-c(t)) s} ds, \\
	&K_z \delta_z \|\phi\|_\sigma e^{(\beta_z-\sigma(t)) t} \int_{t}^{\infty}e^{(\sigma(s)-\beta_z) s} ds\bigg\}.
	\end{align}
	Taking the supremum over $t \in \mathbb{R}$ and applying \eqref{tleq0} and \eqref{tgeq0} to the first term in the $Y$ component gives a form to which we can apply Lemma \ref{lemma1}. The result is 
	\begin{align}
	\sup_{t \in \mathbb{R}}e^{-\sigma(t) t}\|\mathcal{T}(\phi(t),y_0)\| \leq K_y|y_0| + \|\phi\|_\sigma\max\bigg\{ \frac{K_y \delta_y}{\sigma(t) - c(t)}, \frac{K_z \delta_z}{\beta_z - \sigma(t)}, \frac{K_x \delta_x}{\sigma(t) - \alpha_x} \bigg\} < \infty.
	\end{align}
	
	Next, we show that $\mathcal{T}(\phi(t),y_0)$ is continuous in $t$. We split the proof into six cases: $t < 0$, $t > 0$, and $t = 0$ as $d \rightarrow t^+$ and $d \rightarrow t^-$. In each case, we assume that $d$ starts in a small enough ball around $t$ that it matches the sign of $t$. First, we consider $t > 0$ and take $d \rightarrow t^+$: 
	\begin{align}
	\mathcal{T}(\phi(t),y_0)& - \mathcal{T}(\phi(d),y_0) =[I-e^{(d-t)A}]\int_{-\infty}^te^{(t-s)A}F(\phi(s))ds - \int_t^d e^{(d-s)A} F(\phi(s)) ds \\
	&+(e^{tB}-e^{dB})y_0 + [I-e^{(d-t)B}]\int_0^t e^{(t-s)B}G(\phi(s))ds - \int_t^d e^{(d-s)B}G(\phi(s))ds \\
	&- \bigg[\int_t^d e^{(t-s)C}H(\phi(s))ds + [I-e^{(d-t)C}]\int_d^\infty e^{(t-s)C}H(\phi(s))ds\bigg] \\
	=&\underbrace{(e^{tB}-e^{dB})y_0}_\text{(I)} -\underbrace{\sum_{n=1}^{\infty}\frac{((d-t)B)^n}{n!}\int_0^t e^{(t-s)B}G(\phi(s))ds}_\text{(II)} - \underbrace{\int_t^d e^{(d-s)B}G(\phi(s))ds}_\text{(III)} \\
	&- \bigg[\underbrace{\int_t^d e^{(t-s)C}H(\phi(s))ds}_\text{(IV)} -\underbrace{\sum_{n=1}^{\infty}\frac{((d-t)C)^n}{n!}\int_d^\infty e^{(t-s)C}H(\phi(s))ds}_\text{(V)}\bigg] \\
	&-\underbrace{\sum_{n=1}^{\infty}\frac{((d-t)A)^n}{n!}\int_{-\infty}^te^{(t-s)A}F(\phi(s))ds}_\text{(VI)} - \underbrace{\int_t^d e^{(d-s)A} F(\phi(s)) ds}_\text{(VII)}.
	\end{align}
	As $d \rightarrow t^+$, $e^{dB} \rightarrow e^{tB}$ and (I) will approach zero. In terms (III), (IV), and (VII), as $d \rightarrow t^+$, the bounds on the integrals contract and each integral approaches zero. In (II), the bounds on the integral are finite and thus the integral will remain bounded while $\sum_{n=1}^{\infty}\frac{((d-t)B)^n}{n!}$ will approach zero as $d \rightarrow t^+$, forcing the term to zero. The summation terms in (V) and (VI) will also converge to zero. The indefinite integrals are bounded by the boundedness of the norm of the $\mathcal{T}$ map established in the first part of the proof, and thus the terms (V) and (VI) will approach zero as $d \rightarrow t^+$ and the limit as $d \rightarrow t^+$ of this expression will be zero. The same reasoning applies to all further cases. 
\end{proof}

We have now proved that the $\mathcal{T}$ is a well-defined operator.  The next step is the core of the construction---$\mathcal{T}$ is a contraction mapping.

\begin{proposition}\label{lipcont}
	Given assumptions A\ref{assumption: exponential trichotomy}, A\ref{assumption: nonlinear lipschitz}, A\ref{assumption: gap condition}, C\ref{c1}, and C\ref{c2}, for fixed $y_0 \in Y$, $\mathcal{T}(\cdot,y_0)$ is a a contraction mapping.
\end{proposition}
\begin{proof}
	Let $\phi_1$, $\phi_2 \in \mathcal{F}_\sigma$ and denote $\mathcal{T}(\phi_1(t),y_0):=\mathcal{T}(\phi_1(t))$ for a fixed $y_0 \in Y$. Take the norm of $\mathcal{T}(\phi_1(t)) - \mathcal{T}(\phi_2(t))$, then apply A\ref{assumption: exponential trichotomy} and A\ref{assumption: nonlinear lipschitz} to get
	\begin{align}
	\|\mathcal{T}(\phi_1(t)) - \mathcal{T}(\phi_2(t))\| \leq \max\bigg\{ K_x\delta_x\int_{-\infty}^t e^{(t-s)\alpha_x} \|\phi_1(s) - \phi_2(s)\|ds, \\ K_y\delta_y\int_0^t e^{(t-s)c(t)} \|\phi_1(s) - \phi_2(s)\|ds, K_z\delta_z\int_t^\infty e^{(t-s)\beta_z} \|\phi_1(s) - \phi_2(s)\|ds \bigg\}.
	\end{align}
	Multiply by $e^{-\sigma(t)t}$ and take the supremum over $s$ for each $e^{-\sigma(s)s}\|\phi_1(s) - \phi_2(s)\|$ term. Also, take the supremum over $t$ on each side of the expression:
	\begin{align}
	\sup_{t\in\mathbb{R}}e^{-\sigma(t)t}\|\mathcal{T}(\phi_1(t)) - \mathcal{T}(\phi_2(t))\| \leq \|\phi_1 - \phi_2\|_\sigma\max\bigg\{&\sup_{t\in\mathbb{R}}K_x\delta_xe^{(\alpha_x-\sigma(t))t}\int_{-\infty}^t e^{(\sigma(s)-\alpha_x)s}ds,\\
	\sup_{t\in\mathbb{R}}K_y\delta_ye^{(c(t)-\sigma(t))t}\int_0^t e^{(c(t)-\sigma(s))s}ds,
	&\sup_{t\in\mathbb{R}}K_z\delta_ze^{(\beta_z-\sigma(t))t}\int_t^\infty e^{(\sigma(s)-\beta_z)s} ds \bigg\}.
	\end{align}
	
	By Lemma \ref{lemma1}, this simplifies to 
	\begin{align}
	\|\mathcal{T}(\phi_1) - \mathcal{T}(\phi_2)\|_\sigma \leq \delta_\phi\|\phi_1 - \phi_2\|
	\end{align}
	where $\delta_\phi < 1$ is the Lipschitz constant.

\end{proof}

By the Banach Fixed Point Theorem in \cite{banach}, there exists a unique $\phi^* \in \mathcal{F}_\sigma$ such that $\phi^*(t,y_0)= \mathcal{T}(\phi^*(t),y_0)$ for fixed $y_0 \in Y$.  Since $\phi^*$ will play an important role in constructing the center manifold, we investigate some properties about it.  We show that the $\phi^*$ is a unique solution in $\mathcal{F}_\sigma$ to the original system.

\begin{proposition}\label{ufp}
	The fixed point of $\mathcal{T}(\cdot,y_0)$, denoted by $\phi^*(t,y_0)$, is characterized as the unique element in the function space that is the solution to \eqref{bigsystem} with initial condition $\phi(0,y_0)$.
\end{proposition}
\begin{proof}
	We show that $\phi^*(t,y_0)$ is a solution to the system. We start by taking derivatives with respect to $t$: 
	\begin{align}
	\dot{x} &= F(x,y,z) + A\int_{-\infty}^t e^{(t-s)A}F(\phi(s))ds  \\
	\dot{y} &= G(x,y,z) + Be^{tB}y_0 + B\int_0^t e^{(t-s)B}G(\phi(s))ds \\
	\dot{z} &= H(x,y,z) - C\int_t^\infty e^{(t-s)C}H(\phi(s))ds 
	\end{align}
	From the map, $x = \int_{-\infty}^t e^{(t-s)A}F(\phi(s))ds$, $y = e^{tB}y_0 + \int_0^t e^{(t-s)B}G(\phi(s))ds$, and $z = - \int_t^\infty e^{(t-s)C}H(\phi(s))ds$. Substituting in yields the system in \eqref{bigsystem} with the given initial condition, which is unique given the choice of $\phi^*$.
\end{proof}

\begin{definition}
	The center manifold is $\mathcal{M}_c := \{u_0 \in E: u(t,u_0) \in \mathcal{F}_\sigma \}.$
\end{definition}

We show the invariance of the center manifold.
\begin{proposition}
	$\mathcal{M}_c$ is invariant: if $u_0 \in \mathcal{M}_c$, then $u(t_0,u_0) \in \mathcal{M}_c$ for fixed $t_0$.
\end{proposition}
\begin{proof}
	Take $u_0 \in \mathcal{M}_c$ such that $u(t, u_0) \in \mathcal{F}_\sigma$. Fix $t_0$. We need that $u_1 := u(t_0,u_0) \in \mathcal{M}_c$. Since $\mathcal{M}_c = \{u_0 \in E: u(t,u_0) \in \mathcal{F}_\sigma \}$, it remains to show that $u(t,u_1) \in \mathcal{F}_\sigma$. Then, by the fact that $u$ is autonomous, $u(t,u_1) = u(t,u(t_0,u_0))= u(t+t_0,u_0) \in \mathcal{F}_\sigma$. 
\end{proof}

Next, we want to characterize the manifold in terms of $\phi^*$. Let $\Phi$ be the map defined such that $\Phi: Y \rightarrow X \times Z$ by $\Phi(y_0) = \phi(0,y_0)|_{X \times Z}$. We show the following set equivalence. 

\begin{proposition}\label{uiscenter}
	Given Given A\ref{assumption: exponential trichotomy}, A\ref{assumption: nonlinear lipschitz}, A\ref{assumption: gap condition}, C\ref{c1}, and C\ref{c2}, we have $$
	\mathcal{M}_c = \{u_0 : \Phi(y_0) = x_0 + y_0 \} = \emph{graph}{ \Phi}. $$
\end{proposition}
\begin{proof}
	This is a direct result of Proposition \ref{ufp}.
\end{proof}

In the next proof, we use Gronwall's inequality, a proof of which can be found in \cite{gronwall}:
\begin{lemma} {Gronwall's Inequality:}
	If $u(t) \leq p(t) + \int_{t_0}^t q(s) u(s) ds$ for functions $u, p$, and $q$ such that $u$ and $q$ are continuous and $p$ is non-decreasing, then
	\begin{align}\label{gronwall}
	u(t) \leq p(t)\exp\bigg(\int_{t_0}^t q(s)ds \bigg).
	\end{align}
\end{lemma}

\begin{proposition}\label{prop: center lipschitz}
	Given A\ref{assumption: gap condition} and C\ref{c2}, fix $y_1$ and $y_2 \in Y$. Then, $\mathcal{M}_c$ is Lipschitz continuous with Lipschitz constant $K_ye^{\delta_\phi}$: $$\|\Phi(y_1) - \Phi(y_2)\| \leq K_ye^{\delta_\phi} \|y_1 - y_2\|.$$ 
\end{proposition}
\begin{proof}
	Let $\phi_1^*(t)=T(\phi_1^*(t),y_1)$ and $\phi_2^*(t) = T(\phi_2^*(t),y_2)$ and note that by Proposition \ref{uiscenter}:
	\begin{align}
	\|\Phi(y_1) - \Phi(y_2)\| &= \|\phi_1^*(0)|_{X \times Z} - \phi_2^*(0)|_{X \times Z}\|  \leq \|\phi_1^* - \phi_2^*\|_\sigma.
	\end{align}
	This implies that we get a bound for $\|\Phi(y_1) - \Phi(y_2)\|$ if we bound $\|\phi_1^* - \phi_2^*\|_\sigma$. We calculate the difference using the equivalence to the $\mathcal{T}$ map. First notice that 
	\begin{align}
	e^{-\sigma(t)}\|\phi_1^*(t) - \phi_2^*(t)\|  \leq 	\max\bigg\{K_x\delta_x e^{(\alpha_x-\sigma(t)) t} \int_{-\infty}^t e^{(\sigma(s)-\alpha_x)s}e^{-\sigma(s) s} \|\phi_1^*(s) - \phi_2^*(s)\|ds, \\ K_y\delta_y e^{(c(t)-\sigma(t)) t}\int_0^t e^{(\sigma(s)-c(t))s}e^{-\sigma(s) s} \|\phi_1^*(s) - \phi_2^*(s)\|ds, \\
	K_z\delta_z e^{(\beta_z-\sigma(t)) t} \int_t^\infty e^{(\sigma(s)-\beta_z)s}e^{-\sigma(s) s} \|\phi_1^*(s) - \phi_2^*(s)\|ds\bigg\} + K_ye^{(c(t)-\sigma(t)) t}\|y_1 - y_2\|. 
	\end{align}
	We use \eqref{gronwall} in each component to get
	\begin{align}
	e^{-\sigma t}\|\phi_1^*(t) - \phi_2^*(t)\| \leq K_y\|y_1 - y_2\|\max\bigg\{	e^{(c(t)-\sigma(t)) t}\exp(K_x\delta_x e^{(\alpha_x-\sigma(t)) t} \int_{-\infty}^t e^{(\sigma(s)-\alpha_x)s}ds), \\ e^{(c(t)-\sigma(t)) t}\exp(K_y\delta_y e^{(c(t)-\sigma(t)) t}\int_0^t e^{(\sigma(s)-c(t))s}ds), \\
	e^{(c(t)-\sigma(t)) t}\exp(K_z\delta_z e^{(\beta_z-\sigma(t)) t} \int_t^\infty e^{(\sigma(s)-\beta_z)s}ds)\bigg\}. 
	\end{align}
	By Lemma \ref{lemma1}, this simplifies to $\|\phi_1^* - \phi_2^*\|_\sigma \leq K_ye^{\delta_\phi}\|y_1 - y_2\|.$
	
\end{proof}
Now that we have that the manifold is Lipschitz, we have completed the final step in proving the main theorem of this section:

\begin{theorem}\label{thm1}
	Given Given A\ref{assumption: exponential trichotomy}, A\ref{assumption: nonlinear lipschitz}, A\ref{assumption: gap condition}, C\ref{c1}, and C\ref{c2} there exists a unique Lipschitz map $\Phi:Y\rightarrow X \times Z$ such that $\mbox{graph}(\Phi)$ is the center manifold of \eqref{bigsystem}.
\end{theorem}

\section{Regularity of the Manifold}\label{sec:regularity of the manifold}

In this section, we will show that the map $\Phi \in \mathcal{C}^1(Y, X\times Z)$.  To do this, we need additional assumptions regarding the nonlinear terms:
\begin{assumption}\label{assumption: nonlinear c1}
	$F(x,y,z) \in \mathcal{C}^1(E,X)$, $G(x,y,z) \in \mathcal{C}^1(E,Y)$, and $H(x,y,z) \in \mathcal{C}^1(E,Z)$.
\end{assumption}
\begin{assumption}\label{assumption: nonlinear lip}
    For $u_1, u_2 \in E$,
    \begin{align}
    \|DF(u_1) -DF(u_2)\| &\leq \gamma_x\|u_1 -u_2\| \\
    \|DG(u_1) -DG(u_2)\| &\leq \gamma_y\|u_1 - u_2\| \\
    \|DH(u_1) -DH(u_2)\| &\leq \gamma_z\|u_1 -u_2\|,
    \end{align}
    where $\gamma_x, \gamma_y$, and $\gamma_z \in \mathbb{R}_{> 0}$.
\end{assumption}
\noindent 	

With A\ref{assumption: nonlinear lipschitz} and a theorem from \cite{federer}, we know that the norm of the derivative of each nonlinear term is uniformly bounded: for example, given A\ref{assumption: nonlinear lipschitz} and A\ref{assumption: nonlinear c1}, $\|DF(x,y,z)\| \leq \delta_x$. By Proposition \ref{ufp}, $\phi^*(t,y_0)= \mathcal{T}(\phi^*(t),y_0)$ for fixed $y_0$.  Denote this by $\phi_0(t) := \phi^*(t,y_0)$ when $y_0$ is understood. 
We study the derivative of the $\phi$ map using the Banach Fixed Point Theorem. We define the space
\begin{align}
{F}_{1,\sigma} = \big\{\Delta \in C(\mathbb{R},\mathcal{L}(Y,E)): \smash{\displaystyle\sup_{t \in \mathbb{R}}} \mbox{ } e^{-\sigma(t)} \|\Delta(t)\|_{\mathcal{L}(Y,E)} = \|\Delta\|_{1,\sigma} < \infty \big\},
\end{align}
where $\sigma(t)$ is defined as before in \eqref{eq:sigma}. We differentiate $\phi_0 = \mathcal{T}(\phi_0)$ with respect to $y$ to get $\mathcal{T}_1: \mathcal{F}_{1,\sigma} \times Y \rightarrow \mathcal{F}_{1,\sigma}$	where
\begin{align}\label{t1}
\mathcal{T}_1(\Delta(t),y_0) = \underbrace{e^{tB} + \int_0^t e^{(t-s)B}DG(\phi_0(s))\Delta(s)ds}_\text{$Y$}
- \underbrace{\int_{-\infty}^t e^{(t-s)C}DH(\phi_0(s))\Delta(s) ds}_\text{$Z$} \nonumber \\ 
+ \underbrace{\int_t^\infty e^{(t-s)A}DF(\phi_0(s))\Delta(s) ds}_\text{$X$}.
\end{align}

As in the previous section, we show that $\mathcal{T}_1$ is well-defined. This proof is similar to the proof for $\mathcal{T}$ and thus we leave out several details.

\begin{proposition}
	Given A\ref{assumption: exponential trichotomy}, A\ref{assumption: nonlinear lipschitz}, and C\ref{c1}, 
	$\mathcal{T}_1(\Delta(t),y_0)$ is well-defined.
\end{proposition}
\begin{proof}
	First, we show $\mathcal{T}_1: \mathcal{F}_{1,\sigma} \times Y \rightarrow \mathcal{F}_{1,\sigma}$ by showing that $\|\mathcal{T}_1(\Delta,y_0)\|_{1,\sigma} < \infty$ for any $\Delta \in \mathcal{F}_{1, \sigma}$ and $y_0 \in Y$. We obtain the form
	\begin{equation}\label{intermediate}
	\begin{aligned} 
	e^{-\sigma(t) t}\|\mathcal{T}_1(\Delta(t),y_0)\| \leq \max&\bigg\{K_x\delta_x e^{(\alpha_x-\sigma(t)) t}\int_{-\infty}^t e^{(\sigma(s)-\alpha_x) s}e^{-\sigma(s) s}\|\Delta(s)\| ds, \\
	K_y e^{(c(t)-\sigma(t)) t} + &K_y\delta_y e^{(c(t)-\sigma(t)) t}\int_0^t e^{(\sigma(s)-c(t)) s}e^{\sigma(s) s}\|\Delta(s)\| ds, \\
	&K_z\delta_z e^{(\beta_z-\sigma(t)) t}\int_t^\infty e^{(\sigma(s)-\beta_z) s}e^{-\sigma(s) s}\|\Delta(s)\| ds\bigg\}.
	\end{aligned} 
	\end{equation}
	From Proposition \ref{prop: well-defined}, \eqref{intermediate} will simplify down to a set of finite constants. Showing that $\mathcal{T}_1(\Delta(t),y_0)$ is continuous in $t$ also follows Proposition \ref{prop: well-defined}.
\end{proof}

The next step is to show that $\mathcal{T}_1(\cdot,y_0)$ is a contraction mapping.

\begin{proposition}
	Given Given A\ref{assumption: exponential trichotomy} and A\ref{assumption: nonlinear lipschitz}, 
	$\mathcal{T}_1(\Delta(t),y_0)$ is a contraction mapping with rate $\delta_\phi.$
\end{proposition}
\begin{proof}
	Take the difference $\mathcal{T}_1(\Delta_1(t),y_0) - \mathcal{T}_1(\Delta_2(t),y_0)$,
	where $\Delta_1, \Delta_2$ are arbitrary functions in $\mathcal{F}_{1,\sigma}$. Then, we take norms and apply A\ref{assumption: exponential trichotomy} and A\ref{assumption: nonlinear lipschitz}:
	\begin{align}
	e^{-\sigma(t) t}\|\mathcal{T}_1(\Delta_1(t),y_0) - \mathcal{T}_1(\Delta_2(t),y_0)\| \leq   \max\bigg\{K_x\delta_x e^{(\alpha_y - \sigma(t))t}\int_{t}^{\infty}e^{(\sigma(s)-\alpha_y)s}ds, \\ K_y\delta_y e^{(c(t) - \sigma(t))t}\int_{0}^{t}e^{(\sigma(s)-c(t))s}ds, K_z\delta_z e^{(\beta_z - \sigma(t))t}\int_{-\infty}^{t}e^{(\sigma(s)-\beta_z)s}ds\bigg\}\|\Delta_1-\Delta_2\|_{1,\sigma}.
	\end{align}
	The result follows by Lemma \ref{lemma1}.
	
\end{proof}

Let $\Delta^*(t):=\mathcal{T}_1(\Delta^*(t),y_0)$ be the fixed point of the $\mathcal{T}_1$ map given $y_0$. We need the following lemma:

\begin{lemma} \label{bound on phi}
	Given A\ref{assumption: exponential trichotomy}, A\ref{assumption: nonlinear lipschitz}, A\ref{assumption: gap condition}, C\ref{c1}, C\ref{c2}, then let $\phi_1(t):= \phi^*(t,y_1)$ and $\phi_2(t):= \phi^*(t,y_2)$. We have the following bound on $\|\phi_1(t) - \phi_2(t)\|$:
	\begin{align}
	\|\phi_1(t) - \phi_2(t)\| \leq \begin{cases}
	K_y e e^{(K_y \delta_y + \alpha_y)t}\|y_1 - y_2\| & \mbox{ when } t\geq 0 \\
	K_y e e^{(\beta_y - K_y \delta_y)t}\|y_1 - y_2\| & \mbox{ when } t\leq 0.
	\end{cases}
	\end{align}
\end{lemma}
\begin{proof}		
	From Proposition \ref{prop: center lipschitz}, we have that $$\|\phi_1(t) - \phi_2(t) \|e^{-\sigma_n t} \leq K_ye^{\delta_\phi} \|y_1 - y_2\|$$ where $K_ye^{\delta_\phi} = K_y \max\big\{e^{\frac{K_x \delta_x}{\sigma_n - \alpha_x}}, e^{\frac{K_y \delta_y}{\beta_y - \sigma_n}}, e^{\frac{K_z \delta_z}{\beta_z - \sigma_p}}\big\}< K_ye$ in the case that $t \leq 0$. So $$\|\phi_1(t) - \phi_2(t) \| \leq K_y ee^{\sigma_n t} \|y_1 - y_2\|.$$ From C\ref{c2} on $\sigma_n$, we have that $\alpha_x + K_x \delta_x < \sigma_n < \beta_y - K_y \delta_y$. If we multiply through by $t \leq 0$, we get that $(\beta_y - K_y \delta_y)t < \sigma_n t < (\alpha_x + K_x \delta_x) t$. We get the most precise bound on $\|\phi(t,y_1) - \phi(t,y_2) \|$ by letting $\sigma_n \rightarrow \beta_y - K_y \delta_y$. When $t \leq 0$, $$\|\phi_1(t) - \phi_2(t) \| \leq K_y e e^{(\beta_y - K_y \delta_y)t}\|y_1 - y_2\|.$$ 
	
	When $t \geq 0$, 
	\begin{align} 
	\|\phi_1(t) - \phi_2(t) \| \leq K_ye^{\delta_\phi} e^{\sigma_p t} \|y_1 - y_2\|.
	\end{align} 
	As before, from C\ref{c2}, we get that $(\alpha_y + K_y \delta_y)t < \sigma_p t < (\beta_z - K_z \delta_z)t$. Taking $\sigma_p \rightarrow (\alpha_y + K_y \delta_y)$ gives that $$\|\phi_1(t) - \phi_2(t)\| \leq K_y ee^{(K_y \delta_y + \alpha_y)t}\|y_1 - y_2\|.$$
	
	Continuing, we use the following shorthand:
	\begin{align}
	k(t) = \begin{cases}
	- K_y\delta_y & \mbox{ when } t \leq 0 \\ K_y\delta_y & \mbox{ when } t \geq 0,
	\end{cases}
	\end{align}
	and $v(t) = c(t) + k(t)$.
\end{proof}

We show next that $\Delta^* = \partial \phi_0/\partial y$. This follows a similar idea to the proof presented in \cite{castaneda}.

\begin{proposition}\label{prop:DPhi equivalence}
	Given A\ref{assumption: exponential trichotomy}, A\ref{assumption: nonlinear lipschitz}, A\ref{assumption: gap condition}, C\ref{c1}, and C\ref{c2} 
	we have that $\partial \phi(y_0)/\partial y = \Delta^*$, where $\Delta^*(t):=\mathcal{T}_1(\Delta^*(t),y_0)$ and $\phi^*(t,y_0) = \mathcal{T}(\phi^*(t),y_0)$ for a given $y_0 \in Y$.
\end{proposition}
\begin{proof}
	For clarity of notation in this proof, we write out $\phi^*(t,y_0)$. To get that $\partial \phi_0/ \partial y = \Delta^*$, we use the representation of $\phi^*(t,y_0)$ as a fixed point of the $\mathcal{T}$ map and differentiate with respect to $y$. $\Delta^*(t) \in \mathcal{C}(\mathbb{R}, \mathcal{L}(Y,E))$ and therefore is bounded and linear in $Y$. Then, if 
	\begin{align}\label{eq:frechet diff}
	\lim_{h \rightarrow 0} \frac{\| \phi^*(y_0 + h) - \phi^*(y_0) - \Delta^*h\|_\sigma}{\|h\|} = 0
	\end{align}
	where $h \in Y$, we have that $\phi^*(y_0)$ is Fr\'echet differentiable where $\partial \phi/\partial y =\Delta^*$.
	
	First, let 
	\begin{align}
	\rho(t,y_0,h) &= \frac{\| \phi^*(t,y_0 + h) - \phi^*(t,y_0) - \Delta^*(t)h\|_E}{\|h\|}
	\end{align}
	where $t \in \mathbb{R}$ and $y_0$, $h \in Y$. Consider this as $$\rho(t,y_0,h) = \max\{\rho_X(t,y_0,h), \rho_Y(t,y_0,h), \rho_Z(t,y_0,h) \}$$ where, for example,
	\begin{align}
	\rho_X(t,y_0,h) &= \frac{\| \phi^*(t,y_0 + h)|_x - \phi^*(t,y_0)|_x - \Delta^*(t)h|_x \|}{\|h\|}.
	\end{align}
	The result follows if $\sup_{t \in \mathbb{R}} e^{\sigma(t)t}\rho(t,y_0,h) \rightarrow 0$ as $h \rightarrow 0$. We will find the following shorthand to be useful: let $\zeta(s) := \phi^*(s,y_0)$ and $w(s,h) := \phi^*(s,y_0+h) - \phi^*(s,y_0)$.

	Step 1. We have the following estimates: 
	\begin{flalign}
	\rho_X(t,y_0,h) \leq & K_x K_y e e^{\alpha_x t} \int_{-\infty}^{t} e^{(v(t) - \alpha_x)s} R_X(\zeta(s), w(s,h))ds + K_x \delta_x \int_{-\infty}^t e^{(t-s)\alpha_x}\rho(s,y_0,h)ds& \\
	\rho_Y(t,y_0,h) \leq& K_y^2 e e^{c(t) t} \int_0^t e^{k(s) s} R_Y(\zeta(s), w(s,h)) ds + K_y \delta_y \int_0^t e^{(t-s)c(t)} \rho(s,y_0,h)ds, & \\
	\rho_Z(t,y_0,h) \leq &K_z K_y e e^{\beta_z t} \int_{t}^{\infty} e^{(v(t)  - \beta_z)s} R_Z(\zeta(s), w(s,h))ds+ K_z \delta_z \int_t^{\infty} e^{(t-s)\beta_z}\rho(s,y_0,h)ds,
	\end{flalign} 
	where, for example, 
	\begin{equation}\label{eq: uniform bound}
		\begin{aligned} 
		R_X(\zeta(s),w(s,h)) &= \frac{\|F(\zeta(s) + w(s,h)) - F(\zeta(s)) - DF(\zeta(s))w(s,h)\|}{\|w(s,h)\|_E} \leq 2\delta_x .
		\end{aligned} 
		\end{equation}
		
	Step 2. From Step 1, we obtain 
	\begin{align}
	\sup_{t \in \mathbb{R}} e^{-\sigma(t)t} \rho(t,y_0,h) \leq (1- \delta_\phi)^{-1} R(\zeta(t), w(t))
	\end{align} 
	where \begin{align}
	R(\zeta(t), w(t)) = \max\bigg\{\sup_{t \in \mathbb{R}}K_y^2 e e^{(c(t) - \sigma(t)) t} \int_0^t e^{k(s) s} R_Y(\zeta(s), w(s,h)) ds,\\
	\sup_{t \in \mathbb{R}}K_x K_y e e^{(\alpha_x - \sigma(t)) t} \int_{-\infty}^{t} e^{(v(t) - \alpha_x)s} R_X(\zeta(s), w(s,h))ds,\\ 
	\sup_{t \in \mathbb{R}}K_z K_y e e^{(\beta_z - \sigma(t)) t} \int_{t}^{\infty} e^{(v(t)  - \beta_z)s} R_Z(\zeta(s), w(s,h))ds \bigg\}.
	\end{align}
	
	Step 3. We have that $\lim_{h \rightarrow 0} R(\zeta(t), w(t)) = 0.$ Then, it follows that $$\lim_{h \rightarrow 0} \sup_{t \in \mathbb{R}} e^{-\sigma(t)t} \rho(t,y_0,h) \leq (1- \delta_\phi)^{-1}\lim_{h \rightarrow 0} R(\zeta(t), w(t)) = 0.$$ 
	
	Proof for Step 1. We show the steps for $\rho_Y$:
	\begin{align}
	\rho_Y(t,y_0,h) =& \frac{\Big\|\int_0^t e^{(t-s)B} \big([G(\zeta(s)+w(s,h))- G(\zeta(s))] - DG(\zeta(s))\Delta^*(s)h\big)ds\Big\|}{\|h\|} .
	\end{align}
	We add and subtract $DG(\zeta(s))w(s,h)$ in the integrand. Then, we apply A\ref{assumption: exponential trichotomy} and multiply by $1 = \|w(s,h)\|_E/\|w(s,h)\|_E$ to get
	\begin{align}
	\rho_Y(t,y_0,h)\leq &K_y\int_0^t e^{(t-s)c(t)}\frac{\|G(\zeta(s) + w(s,h) - G(\zeta(s)) - DG(\zeta(s))w(s,h)\|}{\|w(s,h)\|_E}\frac{\|w(s,h)\|_E}{\|h\|}ds\\
	&+ K_y \int_0^t e^{(t-s)c(t)} \|DG(\zeta(s))\| \frac{\|w(s,h) - \Delta^*(s)h\|_E}{\|h\|} ds.
	\end{align}
	Note that we have $\frac{\|w(s,h) - \Delta^*(s)h\|_E}{\|h\|} = \rho(s,y_0,h)$. Then,		
	\begin{align}
	\rho_Y(t,y_0,h) \leq K_y\int_0^t e^{(t-s)c(t)}R_Y(\zeta(s),w(s,h)\frac{\|w(s,h)\|_E}{\|h\|}ds + K_y \delta_y \int_0^t e^{(t-s)c(t)} \rho(s,y_0,h)ds. 
	\end{align}
	
	By Lemma \ref{bound on phi}, we have $\|w(s,h)\|_E \leq K_y e e^{v(t)s} \|h\|$ and
	\begin{align}
	\rho_Y(t,y_0,h) &\leq K_y^2 e e^{c(t) t} \int_0^t e^{k(s) s} R_Y(\zeta(s), w(s,h)) ds + K_y \delta_y \int_0^t e^{(t-s)c(t)} \rho(s,y_0,h)ds.
	\end{align}
	We use the same steps to get that $\rho_X$ and $\rho_Z$ are bounded. 
	Proof for Step 2. This follows from Lemma \ref{lemma1}; The second integrals that make up $\rho_X$, $\rho_Y$, and $\rho_Z$ are in the form given by the Lemma. The other integrals get rearranged into the $R(\zeta(t),w(t))$ term:
	\begin{align}
	\sup_{t \in \mathbb{R}} e^{-\sigma(t)t} \rho(t,y_0,h) \leq \delta_\phi\sup_{t \in \mathbb{R}} e^{-\sigma(t)t} \rho(t,y_0,h) + R(\zeta(t), w(t)).
	\end{align} 
	Step 2 follows from the fact that $\delta_\phi < 1$, so we can subtract over and divide by the coefficient $(1-\delta_\phi)$.
	
	Proof for Step 3. We show that $(1-\delta_\phi)^{-1}\lim_{h\rightarrow 0}R(\zeta(t),w(t)) = 0.$ First, we take the supremum of $s$ over $R_Y(\zeta(s),w(s,h))$, $R_X(\zeta(s),w(s,h))$, and $R_Z(\zeta(s),w(s,h))$ and move those terms outside of the integral, which we do because $R_X$, $R_Y$, and $R_Z$ are uniformly bounded, as in \eqref{eq: uniform bound}. We also make each supremum over $s$ independent of $t$ by extending it $s \in \mathbb{R}$. Then, 
	\begin{align}
	R(\zeta(t),w(t)) \leq \max\bigg\{\sup_{s \in \mathbb{R}}R_X(\zeta(s),w(s,h))\sup_{t \in \mathbb{R}}K_x K_y e e^{(\alpha_x - \sigma(t)) t} \int_{-\infty}^{t} e^{(v(t) - \alpha_x)s} ds,\\
	\sup_{s \in \mathbb{R}} R_Y(\zeta(s),w(s,h))\sup_{t \in \mathbb{R}}K_y^2 e e^{(c(t) - \sigma(t)) t} \int_0^t e^{k(s) s} ds,\\
	\sup_{s \in \mathbb{R}}R_Z(\zeta(s),w(s,h))\sup_{t \in \mathbb{R}}K_z K_y e e^{(\beta_z - \sigma(s)) t} \int_{t}^{\infty} e^{(v(t)  - \beta_z)s} ds \bigg\}
	\end{align}
	and we can evaluate each term dependent on $t$ to a constant to get that 
	{\small{\begin{align}
	R(\zeta(t),w(t)) \leq \max\bigg\{&\frac{K_x K_y e}{v(t) - \alpha_x}\sup_{s \in \mathbb{R}}R_X(\zeta(s),w(s,h)), \frac{K_y e}{\delta_y}\sup_{s \in \mathbb{R}} R_Y(\zeta(s),w(s,h)),\\ 
	&\frac{K_z K_y e}{\beta_z-v(t) }\sup_{s \in \mathbb{R}}R_Z(\zeta(s),w(s,h))\bigg\}
	\end{align}}}
	$\frac{K_x K_y e}{v(t) - \alpha_x}$ and $\frac{K_z K_y e}{\beta_z-v(t) } $ are positive by A\ref{assumption: gap condition}. Then we take the limit as $h \rightarrow 0$. By the continuity of $w$, $w \rightarrow 0$ as $h \rightarrow 0$. $R_X(\zeta(s),w(s,h))$, $R_Y(\zeta(s),w(s,h))$, and $R_Z(\zeta(s),w(s,h))$ are the differentations of the nonlinear terms, giving that each term goes to zero as $w \rightarrow 0$. We have $\lim_{h \rightarrow 0}R(\zeta(t),w(t)) = 0$ and $\Delta^*$ is the derivative of $\phi^*$ with respect to $y$, and $\partial \phi^*/\partial y = \Delta^*$.
\end{proof}
To move on we need an extra assumption. 
\begin{assumption}{Restriction on the Gap:}\label{a6} 
	From C\ref{c2}, $K_y \delta_y + \alpha_y \leq 0$ and $\beta_y - K_y\delta_y \geq 0$.
\end{assumption}
Now we show that $\Phi \in \mathcal{C}^1(Y, X\times Z)$.

\begin{theorem}\label{thm2}
	Given A\ref{assumption: exponential trichotomy}, A\ref{assumption: nonlinear lipschitz}, A\ref{assumption: gap condition}, C\ref{c1}, and C\ref{c2} 
	the map $\Phi$ whose graph is the center manifold for \eqref{bigsystem} is $C^1(Y, X\times Z)$.
\end{theorem}
\begin{proof}
	We know from Proposition \ref{prop:DPhi equivalence} that $\partial \phi_0 /\partial y = \Delta^*$, and it follows by the definition of $\Phi(y_0) = \phi_0(0)|_{X\times Z}$ that $D\Phi(y_0) = \frac{\partial \phi_0}{\partial y}(0)|_{X\times Z} =  \Delta^*(0)|_{X\times Z}$.
	
	We also need that $D\Phi(y_0)$ is continuous in $y$. For $y_1$ and $y_2 \in Y$, $\Delta_1^*(t), \Delta_2^*(t)$ be the fixed points of $\mathcal{T}_1(\Delta_1^*(t),y_1), \mathcal{T}_1(\Delta_2^*(t),y_2)$, ie, $\Delta_1^* = \mathcal{T}_1(\Delta_1, y_1)$ etc. Then $$\|D\Phi(y_1)-D\Phi(y_2)\| \leq \|\Delta_1^*(0)-\Delta_2^*(0)\| \leq \|\Delta_1^*-\Delta_2^*\|_{1,\sigma}.$$ We just need to check that $\Delta^*$ is continuous in $y$. From here we can apply A\ref{assumption: exponential trichotomy}. Also, note that \begin{align}\|DG(\phi_1(s))\Delta_1^*(s)-DG(\phi_2(s))\Delta_2^*(s)\| \leq &\|DG(\phi_1(s))-DG(\phi_2(s))\|\|\Delta_1^*(s)\| \\&+ \|DG(\phi_2(s))\|\|\Delta_1^*(s)-\Delta_2^*(s)\|
	\end{align}
	where we add and subtract $DG(\phi_2(s))\Delta_1^*(s)$. Then,
	
	\begin{align}
		\|\Delta_1^*(t)-\Delta_2^*(t)\| \leq \max\bigg\{\int_{-\infty}^{t}K_xe^{(t-s)\alpha_x}\|DF(\phi_1(s))-DF(\phi_2(s))\|\|\Delta_1^*(s)\|ds, \\ +\int_{0}^{t}K_ye^{(t-s)c(t)}\|DG(\phi_2(s))\|\|\Delta_1^*(s)-\Delta_2^*(s)\|ds,\\ \int_{t}^{\infty}K_ze^{(t-s)\beta_z}\|DH(\phi_1(s))-DH(\phi_2(s))\|\|\Delta_1^*(s)\|ds \bigg\} \\
		+ \max\bigg\{
		\int_{-\infty}^{t}K_xe^{(t-s)\alpha_x}\|DF(\phi_2(s))\|\Delta_1^*(s)-\Delta_2^*(s)\|ds,\\ \int_{0}^{t}K_ye^{(t-s)c(t)}\|DG(\phi_1(s))-DG(\phi_2(s))\|\|\Delta_1^*(s)\|ds,\\
	    \int_{t}^{\infty}K_ze^{(t-s)\beta_z}\|DH(\phi_2(s))\|\|\Delta_1^*(s)-\Delta_2^*(s)\|ds\bigg\}.
	\end{align}
	
	Next, we multiply through by $e^{-\sigma(t)t}$ and multiply by $e^{\sigma(s)s}e^{-\sigma(s)s}$ in each integral. Also, we apply the bounds on the derivatives of the nonlinear terms to $\|DG(\phi_2(s))\|$, $\|DH(\phi_2(s))\|$, and $\|DF(\phi_2(s))\|$. We can apply Lemma \ref{lemma1}:
	\begin{align}
	\|\Delta_1^*-\Delta_2^*\|_{1,\sigma} \leq \|\Delta_1^*\|_{1,\sigma}(1-\max\bigg\{\frac{K_x \delta_x}{\sigma_n-\alpha_x},\frac{K_y \delta_y}{\beta_y-\sigma_n}, \frac{K_y \delta_y}{\sigma_p-\alpha_y}, \frac{K_z \delta_z}{\beta_z-\sigma_p}	\bigg\})^{-1}\\
	\max\bigg\{\sup_{t \in \mathbb{R}}e^{(\alpha_x-\sigma(t))t}\int_{-\infty}^{t}K_xe^{(\sigma(s)-\alpha_x)s}\|DF(\phi_1(s))-DF(\phi_2(s))\|ds,\\ \sup_{t \in \mathbb{R}}e^{(c(t)-\sigma(t))t}\int_{0}^{t}K_ye^{(\sigma(s)-c(t))s}\|DG(\phi_1(s))-DG(\phi_2(s))\|ds, \\
	\sup_{t \in \mathbb{R}}e^{(\beta_z-\sigma(t))t}\int_{t}^{\infty}K_ze^{(\sigma(s)-\beta_z)s}\|DH(\phi_1(s))-DH(\phi_2(s))\|ds
	 \bigg\}.
	\end{align}
	Applying A\ref{assumption: nonlinear lip} and Lemma \ref{bound on phi} gives
	\begin{align}
	\|\Delta_1^*-\Delta_2^*\|_{1,\sigma} \leq \|\Delta_1^*\|_{1,\sigma}(1-\max\bigg\{\frac{K_x \delta_x}{\sigma_n-\alpha_x},\frac{K_y \delta_y}{\beta_y-\sigma_n}, \frac{K_y \delta_y}{\sigma_p-\alpha_y}, \frac{K_z \delta_z}{\beta_z-\sigma_p}	\bigg\})^{-1}\\
	\|y_1-y_2\|\max\bigg\{\sup_{t \in \mathbb{R}}e^{(\alpha_x-\sigma(t))t}K_xK_y\gamma_xe\int_{-\infty}^{t}e^{(v(t) +\sigma(s)-\alpha_x)s}ds,\\
	\sup_{t \in \mathbb{R}}e^{(c(t)-\sigma(t))t}K_y^2\gamma_ye\int_{0}^{t}e^{(v(t) + \sigma(s)-c(t))s}ds,\\
	\sup_{t \in \mathbb{R}}e^{(\beta_z-\sigma(t))t}K_zK_y\gamma_ze\int_{t}^{\infty}e^{(v(t) +\sigma(s)-\beta_z)s}ds \bigg\}.
	\end{align}
	The $Y$-component evaluates to zero. By A\ref{a6}, the $X$ and $Z$-components evaluate to $\frac{K_x K_y\gamma_x e}{v(t) + \sigma(t) -\alpha_x}$ and $\frac{K_z K_y\gamma_x e}{\beta_z - v(t) - \sigma(t)}$, respectively. Then we have
	
	\begin{align}
	\|\Delta_1^*-\Delta_2^*\|_{1,\sigma} \leq \|\Delta_1^*\|_{1,\sigma}(1-\max\bigg\{\frac{K_x \delta_x}{\sigma_n-\alpha_x},\frac{K_y \delta_y}{\beta_y-\sigma_n}, \frac{K_y \delta_y}{\sigma_p-\alpha_y}, \frac{K_z \delta_z}{\beta_z-\sigma_p}	\bigg\})^{-1}\\
	\|y_1-y_2\|\max\bigg\{ \frac{K_x K_y\gamma_x e}{v(t) + \sigma(t) -\alpha_x}, \frac{K_z K_y\gamma_x e}{\beta_z - v(t) - \sigma(t)} \bigg\}.
	\end{align}

	Then, $\|\Delta_1^*-\Delta_2^*\|_{1,\sigma}\rightarrow 0$ as $y_1 \rightarrow y_2$ and $D\Phi(\cdot)$ is continuous.
	
\end{proof}

\section*{Acknowledgements}
The authors would like to thank the EXTREEMS-QED program at the College of William and Mary for its support, as well as Professor Michael Jolly for many thought-provoking discussions about this topic.  This work was supported by NSF DMS 0955604.

\end{document}